\documentclass{amsart}

\usepackage{hyperref}
\usepackage{cleveref}

\title[Essential codimension]{Remarks on essential codimension}

\author{Jireh Loreaux}

\address{Department of Mathematics and Statistics\\
Southern Illinois University Edwardsville\\
1 Hairpin Dr.\\
Edwardsville, IL\\
62026-1653\\ USA}

\email{jloreau@siue.edu}

\author{P. W. Ng}

\address{Department of Mathematics\\
University of Louisiana at Lafayette\\
217 Maxim Doucet Hall\\
P. O. Box 43568\\
Lafayette, Louisiana\\
70504--3568\\
USA}

\email{png@louisiana.edu}

\newtheorem{thm}{Theorem}[section]
\newtheorem{lem}[thm]{Lemma}

\newtheorem{df}[thm]{Definition}

\newcommand{\A}{\mathcal{A}}
\newcommand{\B}{\mathcal{B}}
\newcommand{\C}{\mathcal{C}}
\newcommand{\D}{\mathcal{D}}
\newcommand{\F}{\mathcal{F}}
\newcommand{\K}{\mathcal{K}}

\newcommand{\Mul}{\mathcal{M}}

\newcommand{\M}{\mathbb{M}}

\numberwithin{equation}{section}

\DeclareMathOperator{\diam}{diam}
\DeclareMathOperator{\supp}{supp}


\begin{document}

\maketitle

\begin{abstract}
We look for generalizations of the Brown--Douglas--Fillmore essential
codimension result, leading to interesting local uniqueness theorems in
KK theory.
We also study the structure of Paschke dual algebras.
\end{abstract}

\section{Introduction}

The notion of \emph{essential codimension} was introduced by
Brown--Douglas--Fillmore (BDF) in their groundbreaking paper
\cite{BDF1} where they classified all essentially normal operators
using Fredholm indices.  Since then, this notion has had manifold
applications (e.g., \cite{BCP+-2006-Agatoeo,CPS-2003-AM}).
This includes, among other
things, an explanation for the mysterious integers appearing
in Kadison's Pythagorean theorem (\cite{Kad-2002-PNASU,Kad-2002-PNASUa,KL-2017-IEOT,Lor-2018}) as well as other Schur--Horn type
results (\cite{BJ-2015-TAMS,Jas-2013-JFA}).

Here is the BDF definition of essential codimension:

\begin{df} (BDF)
Let $P, Q \in \mathbb{B}(l_2)$ be projections such that
$P - Q \in \K$.
The \emph{essential codimension} of $P$ and $Q$ is given by
\[
[P:Q] =_{df}
\begin{cases}
Tr(P) - Tr(Q) & \makebox{ if } Tr(P) + Tr(Q) < \infty\\
Ind(V^*W) & \makebox{ if } Tr(P) = Tr(Q) = \infty,
\makebox{ where } V^*V = W^*W =1,\\
& WW^* = P, \makebox{ } V V^* = Q.
\end{cases}
\]
In the above, ``$Ind$" means Fredholm index.
\label{df:essentialcodimension}
\end{df}

It is not hard to show that, if $Q \leq P$, then essential codimension
reduces to the usual codimension.
Basic properties of essential codimension and their proofs can be
found in \cite{BrownLee}.  We note that, given that $P - Q \in \K$,
the essential codimension essentially measures ``local differences".

A fundamental result on essential codimension which was stated in
\cite{BDF1} (a proof can be found in \cite{BrownLee})  is the following:

\begin{thm}
Let $P, Q \in \mathbb{B}(l_2)$ be projections such that
$P - Q \in \K$.

Then there exists a unitary $U \in \mathbb{C}1 + \K$ such that
$U P U^* = Q$ if and only if $[P:Q] = 0$.
\label{thm:BDF}
\end{thm}

The main goal of this paper is to find generalizations of this result.
We are following the path first travelled on by \cite{BrownLee},
\cite{LeeFirst}, and \cite{LeeJFA2013} (see also, \cite{LinStableUniqueness}
and \cite{DadarEilersAsympUE}).
Lee (\cite{LeeFirst}) observed that essential
codimension is a basic example of $KK^0$, and thus the BDF essential codimension
result (Theorem \ref{thm:BDF}) is connected to powerful uniqueness theorems,
and our goal is to work out some of the operator theoretic consequences.

In Section~\ref{sec:paschke-dual-algebra} we undertake a study of the Paschke dual algebra $\A^d_{\B}$ of $\A$ relative to $\B$ in the context of when $\A$ is a unital separable nuclear C*-algebra and $\B$ is a separable stable C*-algebra.
In this setting we prove a number of results.
We first establish that the Paschke dual algebra is $K_1$-injective (Theorem~\ref{thm:K1injective} and Theorem \ref{thm:K1injectivePart2}) under certain
restrictions on the canonical ideal, which is essential for proving our theorems in Section \ref{sec:essential-codimension}.
We note that the Paschke dual algebra is a unital properly infinite
C*-algebra, and  it is an open problem whether every properly infinite unital
C*-algebra is $K_1$ injective\footnote{See, for example, \cite{BlanchardRohdeRordam}}.
We then prove that the Paschke dual algebra is dual in the sense that $\A$ and $\A^d_{\B}$ are each other's relative commutants in the corona algebra $\C(\B)$, where $\A$ is identified with its image under the Busby map (Theorem~\ref{thm:MorePaschkeDuality}).  This generalizes a remark of Valette (\cite{Valette}).
The key technique throughout this section is the Elliott--Kucerovsky theory of absorbing extensions \cite{ElliottKucerovsky}.

In Section~\ref{sec:essential-codimension} we prove a few theorems (Theorems \ref{thm:Uniqueness0} and \ref{thm:Uniqueness1}) which can be considered as generalizations of BDF's Theorem~\ref{thm:BDF} to the realm of $KK$-theory where the essential codimension is interpreted as an element of $KK^0$, and the unitary which is a compact perturbation of the identity is replaced by the notion of proper asymptotic unitary equivalence due to Dadarlat and Eilers \cite{DadarEilersAsympUE}.
In order to make this abstract notion of essential codimension more concrete, we simply take $\A = \mathbb{C}$ and, with a few modest hypotheses, arrive at a generalization of Theorem~\ref{thm:BDF} that bears true resemblance to it (see Theorem \ref{thm:BDFQuarterway}).

In Section~\ref{sec:TechnicalLemma},  we prove a technical lemma which is
used in one of the main results in a previous section.

In a separate paper,\footnote{J. Loreaux and P. W. Ng, Remarks on essential
codimension:  Lifting projections. Preprint.} we study the connection between
essential codimension and projection lifting.

\section{The Paschke dual algebra}
\label{sec:paschke-dual-algebra}

For a nonunital C*-algebra $\B$, $\Mul(\B)$ and $\C(\B)$ denote the
multiplier and corona algebras of $\B$ respectively.
$\pi : \Mul(\B) \rightarrow \C(\B)$ denotes the natural quotient map.

\begin{df}
Let $\A$ be a unital separable C*-algebra, let $\B$ be a separable
stable C*-algebra, and let
$\phi : \A \rightarrow \Mul(\B)$ be a unital absorbing trivial extension.
The \emph{Paschke dual algebra of $\A$ relative to $\B$} is defined to
be $\A^d_{\B} =_{df} (\pi \circ \phi(\A))' \in \C(\B)$.
Sometimes, to emphasize the map $\phi$, we will use the notation
$\D_{\phi} =_{df} \A^d_{\B}$.
\label{df:PaschkeDualAlgebra}
\end{df}
We note that $\A^d_{\B}$ is, up to *-isomorphism, independent of
$\phi$.  However, the map $\phi$ is quite important, and in many
treatments of Paschke duality,
one has ``$\phi$" in the notation.  Hence, we also use the alternate
notation ``$\D_{\phi}$".
There is also a definition for nonunital $\A$, but we focus on the unital
case where the definition is simpler (essentially Paschke's and Valette's
original definition).  We so name the Paschke dual algebra because of
Paschke duality, which asserts the existence of group isomorphisms
$K_j(\A^d_{\B}) \cong KK^{j+1}(\A, \B)$ for $j =0, 1$.  (See
\cite{HigsonPaschkeDual}, \cite{Paschke}, \cite{ThomsenAbsorption},
\cite{Valette}.) We will show below (Theorem \ref{thm:MorePaschkeDuality})
that the Paschke dual algebra is also dual in another sense, thus
generalizing a remark of Valette (\cite{Valette}).

Paschke (\cite{Paschke}) focused on the case where $\B = \K$.
However, many of his assertions and arguments remain true in
general.  Sometimes the modifications are straightforward and other
times they are quite nontrivial.

We  fix a notation from extension theory.  Let $\A, \B$ be
C*-algebras with $\B$ nonunital, and let $\phi, \psi : \A \rightarrow
\C(\B)$ be *-homomorphisms.
We say that $\phi$ and $\psi$ are \emph{unitarily equivalent} and
write
\begin{equation}
\phi \sim \psi
\label{equ:UE}
\end{equation}
if there exists a unitary $u \in \Mul(\B)$ such that
$$\pi(u) \phi(a) \pi(u)^* = \psi(a)$$
for all $a \in \A$.

The argument of the first result is very similar to that
of \cite{Paschke} Lemma 1, but every occurrence of Voiculescu's noncommutative Weyl--von Neumann theorem (\cite{Voiculescu})
is replaced with the Elliott--Kucerovsky theory of absorbing extensions
(\cite{ElliottKucerovsky}).  We go through the proof for the
convenience of the reader, expanding some details.

\begin{lem}
Let $\A$ be a unital separable nuclear C*-algebra, and let $\B$
be a separable stable C*-algebra.

Then we have the following:
\begin{enumerate}
\item[(a)] The unit of
$\A^d_{\B}$ is properly infinite.  In fact,
$1 \sim 1 \oplus 1$ in $\M_2 \otimes \A^d_{\B}$.
\item[(b)] The stable equivalence classes of projections in $\A^d_{\B}$
constitute all of $K_0(\A^d_{\B})$.
\end{enumerate}
\label{lem:K0PaschkeDualAlgebra}
\end{lem}

\begin{proof}
(a):
Let $\phi : \A \rightarrow \Mul(\B)$ be a unital trivial absorbing
extension.  Hence, we may identify
$\A= \pi \circ \phi (\A) \subset \C(\B)$, and we may thus
view $\A$ as a unital C*-subalgebra of $\C(\B)$.
And by \cite{ElliottKucerovsky}, the inclusion map
$\iota : \A \hookrightarrow \C(\B)$ is a unital trivial absorbing
extension. (For triviality, note that
the map $\phi(\A) \rightarrow \Mul(\B) :
\pi \circ \phi(a) \mapsto \phi(a)$ is a
*-homomorphism, and note that we are identifying
$\A = \pi \circ \phi(\A)$.)

We may also identify $\A^d_{\B} = (\pi \circ \phi (\A))' \subseteq
\C(\B)$.

Since $\iota$ is trivial and absorbing
$$\iota \oplus \iota \sim \iota.$$
Therefore, there exists an isometry
$\widetilde{v} \in \M_2 \otimes \Mul(\B)$ such that
$$v ( \iota \oplus \iota ) v^* = \iota \oplus 0$$
where $v =_{df} \pi(\widetilde{v})$.
In particular, we have that
\[
v(x \oplus x )v^* = x \oplus 0
\]
for all $x \in \A$.
Hence, since $\A$ is unital,

\begin{equation}
v^* v = 1 \oplus 1 \makebox{  and  } v v^* = 1 \oplus 0.
\label{equ:July2820183PM}
\end{equation}

From the above, we have that for all $x \in \A$,
\begin{eqnarray*}
v(x \oplus x) & = & (x \oplus 0) v \\
& = & (x \oplus x) v \makebox{   (since $vv^* = 1 \oplus 0$).  }
\end{eqnarray*}
Hence, $v \in \M_2 \otimes \A^d_{\B}$.
From this and (\ref{equ:July2820183PM}),
the unit of $\A^d_{\B}$ is Murray--von Neumann equivalent to two copies
of itself.

(b):  This follows immediately from (a).
\end{proof}

We note that it is an open problem whether every unital properly infinite
C*-algebra is $K_1$ injective \cite{BlanchardRohdeRordam}, and the Paschke dual algebra is an interesting
and important case of this.  We now move towards proving
$K_1$ injectivity under additional hypotheses.

The next lemma ensures that under appropriate conditions, given any unitary $u$ in the commutant of $\A$ (relative to some larger unital algebra), and given a unital trivial absorbing extension, the image of $u$ in the Paschke dual of $\A$ lies in the connected component of the identity in the unitary group.

\begin{lem}
Let $\C$ be a unital C*-algebra and
$\A \subseteq \C$ a separable nuclear unital C*-subalgebra.
Say that $u \in \A'$ ($\subseteq \C$) is a unitary. Let $\B$
be a separable simple stable C*-algebra.
Let
$\phi : C^*(\A, u) \rightarrow \Mul(\B)$ be a unital trivial absorbing
extension.

Then there exists a norm-continuous path of unitaries
$\{ v_t \}_{t \in [0,1]}$
in $(\pi \circ \phi(\A))'$ ($\subseteq \C(\B)$) such that
$v_0 = \pi \circ \phi(u)$ and $v_1 = 1$.
\label{lem:July30201812PM}
\end{lem}

\begin{proof}
Since $\B$ is stable, we may work with $\B \otimes \K$ instead of $\B$.

By the universal property of the maximal tensor product, $C^*(\A, u)$
is a quotient of $\A \otimes_{\max} C(S^1)$, which is nuclear since
$\A$ and $C(S^1)$ are nuclear.  Hence, $C^*(\A, u)$ is a nuclear
C*-algebra.

Since $C^*(\A, u)$ is separable, let $\{ \sigma_n \}_{n=1}^{\infty}$
be a dense sequence in $\widehat{C^*(\A, u)}$ (the space of irreducible
*-representations of $C^*(\A,u)$) such that every
term in $\{ \sigma_n \}$ reoccurs infinitely many times.
Let $\sigma' : C^*(\A, u) \rightarrow \mathbb{B}(l_2)$ be the unital
essential *-representation given by
$$\sigma' = \bigoplus_{n=1}^{\infty} \sigma_n.$$

Then by \cite{KasparovAbsorbing} Theorem 6 (see also \cite{Blackadar}
Theorem 15.12.4 and \cite{ElliottKucerovsky} Theorem 17),
the map
\[
\sigma: C^*(\A, u) \rightarrow \Mul(\B \otimes \K)
: x \mapsto 1_{\Mul(\B)} \otimes \sigma'(x)
\]
is a unital trivial absorbing extension.
Hence, since $\phi$ is also a unital trivial absorbing extension,
there exists a unitary $w \in \Mul(\B \otimes \K)$ such that
$$\phi(x) - w \sigma(x) w^* \in \B \otimes \K$$
for all $x \in C^*(\A, u)$.

Note that for all $n$, since $\sigma_n$ is an irreducible *-representation
of $C^*(\A, u)$, and since $u$ commutes with every element of
$C^*(\A, u)$, $\sigma_n(u) \in S^1$.
So let $\theta_n \in [0, 2 \pi)$ such that
$\sigma_n(u) = e^{i \theta_n}1$.

Now for all $t \in [0,1]$, let
$$v'_t =_{df} w(1 _{\Mul(\B)} \otimes \bigoplus_{n=1}^{\infty} e^{i(1-t)\theta_n}1 ) w^*.$$
And let
$$v_t =_{df} \pi (v'_t).$$

Then $\{ v'_t \}_{t \in [0,1]}$ is a norm continuous path of unitaries
in $w \sigma(\A)' w^*$ ($\subseteq \Mul(\B \otimes \K)$), and
so $\{ v_t \}_{t \in [0,1]}$ is a norm continuous path of unitaries such that
$$v_0 = \pi \circ \phi(u), \makebox{ } v_1 = 1$$
and
$v_t \in (\pi \circ \phi(\A))'$ for all $t \in [0,1]$.
\end{proof}

Recall that for a unital C*-algebra $\D$, $U(\D)$ denotes the unitary
group of $\D$, and $U(\D)_0$ denotes the elements of $U(\D)$ that are
in the connected component of the identity.

We first focus on the case where the canonical ideal is either
$\K$ or simple purely infinite.  It is well-known that
this is exactly the
case with ``nicest" extension theory, since, among other things,
a BDF--Voiculescu type
absorption
result holds.  In fact, in this context, under a nuclearity
hypothesis,
Kasparov's $KK^1$ classifies all essential extensions.

The next result generalizes \cite{Paschke} Lemma 3(2).

\begin{lem}
Let $\A$ be a unital separable nuclear C*-algebra, and $\B$ a separable
stable simple C*-algebra such that either $\B \cong \K$ or $\B$
is purely infinite.

Then the map
$$U(\A^d_{\B})/ U(\A^d_{\B})_0 \rightarrow U(\M_2 \otimes \A^d_{\B})/
U(\M_2 \otimes \A^d_{\B})_0$$
given by
$$[u] \rightarrow [u \oplus 1]$$
is injective.
\label{lem:July30201812:50PM}
\end{lem}

\begin{proof}

Let $\phi : \A \rightarrow \Mul(\B)$ be a unital trivial absorbing
extension.
We may identify $\A^d_{\B} = (\pi \circ \phi(\A))'$.

Let $u \in \A^d_{\B}$ be a unitary such that
$$u \oplus 1 \sim_h 1 \oplus 1$$
in $\M_2 \otimes \A^d_{\B}$.

Let $\sigma : C^*(\pi \circ \phi(\A), u) \rightarrow \Mul(\B)$
be a unital trivial absorbing extension.

Since $\sigma |_{\pi \circ \phi(\A)}$ is a unital trivial absorbing
extension, conjugating $\sigma$ by an appropriate unitary if necessary,
we may assume that
$\pi \circ \sigma(x) = x$ for all $x \in \pi \circ \phi(\A)$.
(After all, by \cite{ElliottKucerovsky},
the map $\pi \circ \phi(\A) \rightarrow \Mul(\B) :
\pi \circ \phi(a) \rightarrow \phi(a)$ is also a unital trivial
absorbing extension.)

By Lemma \ref{lem:July30201812PM}, we have that
\begin{equation}
\pi \circ \sigma(u) \sim_h 1
\label{equ:July30201810AM}
\end{equation}
in $(\pi \circ \sigma(\pi \circ \phi(\A))' = (\pi \circ \phi(\A))'
= \A^d_{\B}$.

Since either $\B \cong \K$ or $\B$ is simple purely infinite,
it follows, by \cite{ElliottKucerovsky} Theorem 17, that
the inclusion map $\iota : C^*(\pi \circ \phi(\A), u) \rightarrow \C(\B)$
is
a unital trivial absorbing extension.  Hence,
$$\iota \oplus (\pi \circ \sigma) \sim \iota.$$

Hence, there exists an isometry $W \in \M_2 \otimes \Mul(\B)$
such that $W^*W = 1 \oplus 1 = 1_{\M_2 \otimes \Mul(\B)}$,
$W W^* = 1 \oplus 0$ and if $w =_{df} \pi(W)$, then
$$w (\iota \oplus (\pi \circ \sigma)) w^* = \iota \oplus 0.$$

As a consequence, we have that
\begin{equation}
w(u \oplus (\pi \circ \sigma(u)))w^* = u \oplus 0,
\label{equ:July20201812:20PM}
\end{equation}

and

\begin{equation}
w(x \oplus x) w^* = x \oplus 0
\label{equ:July30201812:45PM}
\end{equation}
for all $x \in \pi \circ \sigma(\A)$.

Note that by (\ref{equ:July30201812:45PM}),
for all $x \in \pi \circ \sigma(\A)$,
\begin{eqnarray*}
w(x \oplus x) & = & (x \oplus 0) w\\
              & = & (x \oplus x) w \makebox{  (since } w w^* = 1 \oplus 0
\makebox{ )}\\
\end{eqnarray*}

Hence,
$$w \in \M_2 \otimes \A^d_{\B}.$$

Now by (\ref{equ:July30201810AM}),
$$u \oplus (\pi \circ \sigma(u)) \sim_h u \oplus 1$$
in $\M_2 \otimes \A^d_{\B}$.
Also, by the hypothesis on $u$,
$$u \oplus 1 \sim_h 1 \oplus 1$$
in $\M_2 \otimes \A^d_{\B}$.
So
$$u \oplus (\pi \circ \sigma(u)) \sim_h 1 \oplus 1$$
in $\M_2 \otimes \A^d_{\B}$.
Conjugating the continuous path of unitaries by $w$ and
applying (\ref{equ:July20201812:20PM}), we have that
$$u \sim_h 1$$
in $\A^d_{\B}$.

\end{proof}

\begin{thm}
Let $\A$ be a unital separable nuclear C*-algebra and $\B$ a separable
simple stable C*-algebra such that either $\B \cong \K$ or $\B$ is
purely infinite.

Then $\A^d_{\B}$ is $K_1$-injective.
Moreover, for all $n \geq 1$, the map
\[
U(\M_n \otimes \A^d_{\B})/U(\M_n \otimes \A^d_{\B})_0
\rightarrow
U(\M_{2n} \otimes \A^d_{\B})/U(\M_{2n} \otimes \A^d_{\B})_0
\]
given by
\[
[u] \mapsto [u \oplus 1]
\]
is injective.
\label{thm:K1injective}
\end{thm}

\begin{proof}

By Lemma \ref{lem:K0PaschkeDualAlgebra},
we have that the unit of the Paschke algebra $\A^d_{\B}$ satisfies
$1 \oplus 1 \sim 1$.  Hence, for all $n$,
$\A^d_{\B} \cong \M_n \otimes \A^d_{\B}$.
Thus, the result follows from
Lemma \ref{lem:July30201812:50PM}.
\end{proof}

We now move towards understanding $K_1$ injectivity of the Paschke
dual algebra, when the
canonical ideal is no longer elementary nor simple purely infinite.
Outside of these small number of cases, our knowledge of extension theory
is highly incomplete and the questions that arise are much more
challenging.

Let $\D$ be a C*-algebra and $\C \subseteq \D$ a C*-subalgebra.
We say that $\C$ is \emph{strongly full} in $\D$ if every nonzero
element of $\C$ is full in $\D$.  For every nonzero $x \in \D$,
we say that $x$ is \emph{strongly full} in $\D$ if $C^*(x)$ is a strongly
full C*-subalgebra of $\D$.

\begin{lem}
Let $\D$ be a unital C*-algebra and $\A \subseteq \D$ a unital simple
C*-subalgebra.
Suppose that $u \in \A'$ is a strongly full unitary element of $\D$.

Then $C^*(u, \A)$ is strongly full in $\D$.
\label{lem:StronglyFull}
\end{lem}

\begin{proof}
It suffices to prove that every nonzero positive element of
$C^*(u, \A)$ is  full in $\D$.

Let $c \in C^*(u, \A)$ be a nonzero positive element.
Hence, there exists a continuous function
$g : S^1 \rightarrow [0,1]$, and an element $a \in \A_+$ such that
$g(u) a \neq 0$ and $0 \leq g(u) a \leq c$.

Since $\A$ is unital and simple, let $x_1, x_2, ..., x_n \in \A$
be such that
$$\sum_{j=1}^n x_j a x_j^* = 1.$$

Hence,
$$\sum_{n=1}^n x_j g(u) a x_j^* = \sum_{n=1}^n g(u) x_j a x_j^* = g(u).$$

Since $g(u)$ is a full element of $\D$, it follows that
$g(u)a$ is a full element of $\D$.  Hence, $c$ is a full element of $\D$.
Since $c$ was arbitrary, $C^*(u, \A)$ is a strongly full
C*-subalgebra of $\D$.
\end{proof}

Recall that a separable stable C*-algebra $\B$ is said to have the
\emph{corona factorization property} (CFP) if every norm-full projection
in $\Mul(\B)$ is Murray--von Neumann equivalent to $1_{\Mul(\B)}$
(\cite{KucerovskyNgAUE}).

Many C*-algebras have the CFP.  For example, all separable simple
C*-algebras that are either purely infinite or have strict comparison
of positive elements, including all simple C*-algebras classified in
the Elliott program, have the CFP.
In fact, it is quite difficult to construct a simple separable C*-algebra
without CFP.

Recall also, that
a map $\phi : \A \rightarrow \C$ between C*-algebras is said to be
\emph{norm full} or \emph{full}
if for every $a \in \A - \{ 0 \}$, $\phi(a)$ is a full
element of $\C$, i.e., $Ideal(\phi(a)) = \C$.

We say that a *-homomorphism $\phi : \A \rightarrow \Mul(\B)$ \emph{absorbs
0} if $\pi \circ \phi \oplus 0 \sim \pi \circ \phi$.

In \cite{KucerovskyNgAUE}, the following result was proven:
\begin{thm}
Let $\B$ be a separable stable C*-algebra with the CFP,
$\A$ a separable C*-algebra, and $\phi : \A \rightarrow \Mul(\B)/\B$
an essential extension such that either $\phi$ is unital or
$\phi$ absorbs $0$.

Then $\phi$ is nuclearly absorbing if and only if $\phi$ is norm-full.

As a consequence, if, in addition, $\A$ is nuclear, then
$\phi$ is absorbing if and only if $\phi$ is norm-full.

In the above, when $\phi(1) = 1$ and we say that $\phi$ is absorbing, we mean that
$\phi$ is absorbing in the unital sense.
\label{thm:KucerovskyNgAbsorption}
\end{thm}

Let $\B$ be a nonunital separable stable simple C*-algebra
with a nonzero projection $e \in \B$.  We let $T_e(\B)$ denote
the set of all tracial states on $e \B e$.  It is well known
that $T_e(\B)$, with the weak* topology, is
a Choquet simplex.  Moreover, it is also well known that
$\B \cong e \B e \otimes \K$ and that every $\tau \in T_e(\B)$
extends to a trace (which can take the value $\infty$) on
$\Mul(\B)_+$.  If $e' \in \B$ is another nonzero projection,
then $T_e(\B)$ and $T_{e'}(\B)$ are homeomorphic, and
$T_e(\B)$ has finitely many extreme points if and only if
$T_{e'}(\B)$ has finitely many extreme points.  Our results will
be independent of the choice of nonzero projection in $\B$, and
hence, we will write $T(\B)$ to mean $T_e(\B)$ for some $e
\in Proj(\B) - \{ 0 \}$.

Recall that for all $a \in \B_+$ and for all $\tau \in T(\B)$,
$$d_{\tau}(a) =_{df} \lim_{n \rightarrow \infty} \tau(a^{1/n}).$$

Recall that $\B$ is said to have \emph{strict comparison} for positive
elements if for all $a, b \in \B_+$,

$$d_{\tau}(a) < d_{\tau}(b) \makebox{  whenever  }
d_{\tau}(b) < \infty \makebox{ } \forall \tau \in T(\B)
\makebox{  if and only if  } a \preceq b.$$

In the above, $a \preceq b$ means that there exists
$\{ x_k \}$ in $\B$ such that $x_k b x_k^* \rightarrow a.$

In the next proof, we use a key technical lemma, Lemma
\ref{lem:TechnicalLemma}, whose proof we provide in the
later Section \ref{sec:TechnicalLemma}.

\begin{lem}
Let $\A$ be a unital separable simple nuclear C*-algebra, and $\B$ a separable
stable simple C*-algebra with a nonzero projection,
strict comparison of positive elements
and for which $T(\B)$ has finitely many extreme points.

Suppose that there exists a *-embedding $\A \hookrightarrow \B$.

Then the map
$$U(\A^d_{\B})/U(\A^d_{\B})_0  \rightarrow U(\M_2 \otimes A^d_{\B})/
U(\M_2 \otimes\A^d_{\B})_0$$
given by
$$[u] \mapsto [u \oplus 1]$$
is injective.
\label{lem:Oct20201812PM}
\end{lem}

\begin{proof}
By the hypotheses, there exist a sequence $\{ p_n \}_{n=1}^{\infty}$
of pairwise orthogonal projections in $\B$,
a sequence $\{ \phi_n \}_{n=1}^{\infty}$ of *-embeddings from
$\A$ to $\B$, and a sequence $\{ v_{n,1} \}_{n=1}^{\infty}$ of
partial isometries in $\B$ such that the following statements
are true:
\begin{enumerate}
\item $p_m \sim p_n$ for all $m, n$. In fact,
$v_{n, 1}^* v_{n, 1} = p_1$ and $v_{n, 1} v_{n, 1}^* = p_n$ for all
$n$.
\item $\sum_{n=1}^{\infty} p_n = 1_{\Mul(\B)}$, where the sum
converges strictly.
\item $\phi_n(1) = p_n$ for all $n$.
\item $v_{n, 1} \phi_1(x) v_{n, 1}^* = \phi_n(x)$, for all
$x \in \A$ and for all $n$.
\end{enumerate}

Let $\phi : \A \rightarrow \Mul(\B)$ be the unital *-homomorphism
given by
$$\phi =_{df} \sum_{n=1}^{\infty} \phi_n.$$

Then by \cite{LinStableUniqueness} (see also \cite{ElliottKucerovsky}
Theorem 17),  $\pi \circ \phi$ is a unital trivial absorbing extension.
(In the literature, $\phi$ is often called the ``Lin extension".)

We may identify $\A^d_{\B} = (\pi \circ \phi(\A))'.$

Let $u \in \A^d_{\B}$ be a unitary such that
$$u \oplus 1 \sim_h 1 \oplus 1$$
in $\M_2 \otimes \A^d_{\B}$.

By Lemma \ref{lem:TechnicalLemma},
there exists a unitary $v \in \A^d_{\B}$ such that
$$u \sim_h v$$
in $\A^d_{\B}$, and $v$ is strongly full in $\C(\B)$.
Hence, we may assume that $u$ is a strongly full element
of $\C(\B)$.
Hence, by Lemma \ref{lem:StronglyFull},
$C^*(u, \pi \circ \phi(\A))$ is a strongly full unital C*-subalgebra
of $\C(\B)$.

Hence, by Theorem \ref{thm:KucerovskyNgAbsorption},
the inclusion map
$$\iota : C^*(u, \pi \circ \phi(\A)) \hookrightarrow \C(\B)$$
is a unital absorbing extension.

The rest of the proof is exactly the same as that of
Lemma \ref{lem:July30201812:50PM}.
\end{proof}

\begin{thm}
Let $\A$ be a unital separable simple nuclear C*-algebra, and
$\B$ a separable stable simple C*-algebra with a nonzero projection,
strict comparison of positive elements, and for which $T(\B)$ has
finitely many extreme points.

Then $\A^d_{\B}$ is $K_1$-injective.  Moreover, for all $n \geq 1$,
the map
$$U(\M_n \otimes \A^d_{\B})/ U(\M_n \otimes \A^d_{\B})_0
\rightarrow U(\M_{2n} \otimes \A^d_{\B})/ U(\M_{2n} \otimes \A^d_{\B})_0$$
given by
$$[u] \mapsto [u \oplus 1]$$
is injective.
\label{thm:K1injectivePart2}
\end{thm}

\begin{proof}
The proof is exactly the same as that of
Theorem \ref{thm:K1injective}, except that
Lemma \ref{lem:July30201812:50PM} is replaced with
Lemma \ref{lem:Oct20201812PM}.
\end{proof}

We fix a terminology that will only be used in the next theorem.
Let $\A$ be a unital separable nuclear C*-algebra, and
let $\B$ be a separable stable C*-algebra.
Let $\phi : \A \rightarrow \Mul(\B)$ be a unital trivial absorbing
extension.
Recall that we can identify $\A^d_{\B} = (\pi \circ \phi(\A))'$
($\subseteq \C(\B)$).
Since $\pi \circ \phi$ is injective, we may identify $\A$
with $\pi \circ \phi(\A)$.  When $\A$ and $\A^d_{\B}$
sit in $\C(\B)$ in the above manner, we say that $\A$ and
$A^d_{\B}$ are in \emph{standard position} in $\C(\B)$.

\begin{thm}
Let $\A$ be a separable simple unital nuclear C*-algebra, and let
$\B$ be a separable stable simple C*-algebra.
Suppose that $\A$ and $\A^d_{\B}$ are in standard position in
$\C(\B)$.

Then
$$\A' = \A^d_{\B} \makebox{  and  } (\A^d_{\B})' = \A.$$
\label{thm:MorePaschkeDuality}
\end{thm}

\begin{proof}
The first equality follows trivially from the definition of
$\A^d_{\B}$.

The proof of the second equality is exactly the same as that
of \cite{NgDoubleCommutantFinite} Theorem 1. We note that, in our
context,
the inclusion map $\iota : \A \rightarrow \C(\B)$
is a unital trivial absorbing extension.  Hence, the hypothesis,
that $[\iota] \in \mathcal{T}$ (notation as in
\cite{NgDoubleCommutantFinite} Theorem 1)
in \cite{NgDoubleCommutantFinite} Theorem 1 is satisfied.  Also,
since $\iota$ is absorbing, the hypothesis that $\B$ satisfies the
CFP in \cite{NgDoubleCommutantFinite} Theorem 1 is unnecessary.
\end{proof}

Thus, the Paschke dual algebra is ``dual" in still another sense.

\section{Essential codimension}
\label{sec:essential-codimension}

In what follows, we will let $KK$ denote the  generalized
homomorphism picture of KK theory (see, for example, \cite{JensenThomsenBook}
Chapter 4).

In \cite{LeeFirst}, Lee observed that the BDF notion of essential
codimension (Definition \ref{df:essentialcodimension})
is a special case of an element of $KK^0$.  He thus gave the following
definition:

\begin{df}
Let $\B$ be a separable stable C*-algebra, and let $P, Q \in \Mul(\B)$ be
projections such that $P - Q \in \B$.

Let $\phi, \psi : \mathbb{C} \rightarrow \Mul(\B)$ be *-homomorphisms for which
$\phi(1) = P$ and $\psi(1) = Q$.

The \emph{essential codimension} of $P$ and $Q$ is given by
\[
[P:Q] =_{df} [\phi, \psi] \in KK(\mathbb{C}, \B) \cong K_0(\B).
\]
Here, $[\phi, \psi]$ is the class of the generalized homomorphism
$(\phi, \psi)$ in $KK(\mathbb{C}, \B)$.
\label{df:generalizedessentialcodimension}
\end{df}

It is not hard to see (e.g., \cite{LeeJFA2013} Remark 2.2)
that in the case where $\B = \K$, Definition
\ref{df:generalizedessentialcodimension} coincides with the original
BDF essential codimension (Definition \ref{df:essentialcodimension}).
Thus, $KK^0$ concerns the local aspects of operator theory,
as opposed to $KK^1$ which deals with the asymptotic aspects
(e.g., classifying essentially
normal operators up to unitary equivalence modulo the compacts).


Towards generalizing the BDF essential codimension result (Theorem \ref{thm:BDF}),
we recall the notion of proper asymptotic unitary equivalence (see
\cite{DadarEilersAsympUE}).

\begin{df}
  \label{def:proper-asymptotic}
  Let $\A$, $\B$ be C*-algebras, with $\B$ nonunital.
  Let $\phi, \psi : \A \rightarrow \Mul(\B)$ be two *-homomorphisms.
  \begin{enumerate}
  \item $\phi$ and $\psi$
    are said to be \emph{asymptotically unitarily equivalent}
    ($\phi \sim_{asymp} \psi$) if there exists
    a (norm-) continuous path $\{ u_t \}_{t \in [0, \infty)}$
    of unitaries in $\Mul(\B)$ such that for all $a \in \A$,
    \begin{enumerate}
    \item[i.] $\phi(a) - u_t \psi(a) u_t^* \in \B$, for all $t$, and
    \item[ii.]  $\| \phi(a) - u_t \psi(a) u_t^* \| \rightarrow 0$
      as $t \rightarrow \infty$.
    \end{enumerate}
  \item $\phi$ and $\psi$ are said to be \emph{properly asymptotically
      unitarily equivalent} ($\phi \underline{\approx} \psi$) if
    $\phi$ and $\psi$ are asymptotically unitarily equivalent where the
    path of unitaries satisfy that $u_t \in \mathbb{C}1 +\B$ for all $t$.
  \end{enumerate}
\end{df}

We note that proper asymptotic unitary equivalence is a \emph{local notion}.
This is in fitting with the BDF essential codimension theorem.

In \cite{DadarEilersAsympUE}, the following generalization of Theorem \ref{thm:BDF}
was given:  Let $\A, \B$ be separable C*-algebras with $\B$ stable, and let
$\phi, \psi : \A \rightarrow \Mul(\B)$ be *-homomorphisms such that
$\phi(a) - \psi(a) \in \B$.  Then $[\phi, \psi] = 0$ in $KK(\A, \B)$ if and only
if there exists a *-homomorphism $\sigma : \A \rightarrow \Mul(\B)$
such that $\phi \oplus \sigma \underline{\approx} \psi \oplus \sigma$.

We note that \cite{DadarEilersAsympUE} was inspired by and extensively used ideas
from the earlier stable uniqueness paper \cite{LinStableUniqueness}.
We also note that results of the above type can be used to produce (unbounded)
stable uniqueness theorems.  This idea is essentially
due to Lin (\cite{LinStableUniqueness}).

We now introduce and prove our generalization of Theorem \ref{thm:BDF}.
The proof essentially follows that of \cite{LeeFirst} Theorem 2.11 which
follows that of \cite{DadarEilersAsympUE} Theorem 3.12.  As noted above,
\cite{DadarEilersAsympUE} used extensively the ideas of
\cite{LinStableUniqueness}.  In fact, the argument is essentially that of
\cite{LinStableUniqueness}:
A proper asymptotic unitary equivalence
induces a continuous path of
automorphisms on $\phi(\A) + \B$.  Then, following \cite{LinStableUniqueness},
we prove innerness
of the automorphisms.
We sketch the proof for the convenience of the reader.

Recall that $KK$ denotes the generalized homomorphism picture of KK theory
(e.g., see \cite{JensenThomsenBook} Chapter 4).  In the next proof, we will
let $KK_{Higson}$ denote Higson's definition of KK theory (e.g., see
\cite{HigsonKK} Section 2).

Recall that a trivial extension
$\phi$ is said to \emph{absorb the zero extension}
if $\pi \circ \phi \oplus 0 \sim \pi \circ \phi$.

\begin{thm}
Let $\A, \B$ be separable C*-algebras with $\A$ nuclear and $\B$
stable and simple purely infinite.
Let $\phi, \psi : \A \rightarrow \Mul(\B)$ be essential
extensions
such that
$\phi(a) - \psi(a) \in \B$ for all $a \in \A$.

Suppose also that either both $\phi$ and $\psi$ are unital, or both
$\phi$ and $\psi$ absorb the zero extension.

Then $[\phi, \psi] = 0$ in $KK(\A, \B)$ if and only if
$\phi \underline{\approx} \psi$.

\label{thm:Uniqueness0}
\end{thm}

\begin{proof}
The ``if" direction is trivial.

We now prove the ``only if" direction.
Note that by \cite{ElliottKucerovsky} Theorem 17, both $\phi$
and $\psi$ are absorbing extensions.\footnote{Of course, when both
are unital, we mean that they are unitally absorbing.}

Let $\widetilde{\A}$ denote the unitization of $\A$ if $\A$ is nonunital,
and $\A \oplus \mathbb{C}$ if $\A$ is unital.
By \cite{ElliottKucerovsky}, if $\phi : \A \rightarrow \Mul(\B)$ is an
absorbing extension, then the map $\widetilde{\phi} :
\widetilde{\A} \rightarrow \Mul(\B)$ given by
$\widetilde{\phi} |_{\A} = \phi$ and
$\widetilde{\phi}(1) = 1$ is a unital absorbing extension.  Thus, we may assume that
$\A$ is unital and $\phi$ and $\psi$ are unital absorbing trivial
extensions.

As in the previous section, we may identify the Paschke dual algebra
$\A^d_{\B} = (\pi \circ \phi(\A))' \in \C(\B)$.

By \cite{LeeFirst} Theorem 2.5, $\phi \sim_{asym} \psi$.
I.e., there exists a norm continuous path  $\{ u_t \}_{t \in [0, \infty)}$
of unitaries in $\Mul(\B)$ such that
$$u_t \phi(a) u_t^* - \psi(a) \in \B$$
for all $t$ and for all $a \in \A$, and
$$\| u_t \phi(a) u_t^* - \psi(a) \| \rightarrow 0$$
as $t \rightarrow \infty$, for all $a \in \A$.

It is trivial to see that this implies that
$$[\phi, u_0\phi u_0^*] = [ \phi, \psi ] = 0,$$
and that $\pi(u_t) \in (\pi \circ \phi(\A))' = \A^d_{\B}$ for all $t$.

It is well-known that we have
a group isomorphism $KK(\A, \B) \rightarrow KK_{Higson}(\A, \B) : [\phi, \psi]
\rightarrow [\phi, \psi, 1]$.
Hence, $[\phi, u_0 \phi u_0^*, 1] = 0$ in $KK_{Higson}(\A, \B)$.
Hence, by \cite{HigsonKK} Lemma 2.3, $[\phi, \phi, u_0^*] = 0$
in $KK_{Higson}(\A, \B)$.

By Thomsen's Paschke duality theorem (\cite{ThomsenAbsorption} Theorem 3.2),
there is a group isomorphism $K_1(\A^d_{\B}) \rightarrow KK_{Higson}(\A, \B)$
which sends $[\pi(u_0)]$ to $[\phi, \phi, u_0^*]$.
Hence, $[\pi(u_0)] = 0$ in $K_1(\A^d_{\B})$. Hence, by
Theorem \ref{thm:K1injective}, $\pi(u_0) \sim_h 1$ in $\A^d_{\B} =
(\pi \circ \phi(\A))'$.
Hence, there exists a unitary $v \in \mathbb{C}1 + \B$ such that
$v^* u_0 \sim_h 1$ in $\pi^{-1}(\A^d_{\B})$.

Hence, modifying an initial segment of $\{ v^*u_t \}_{t \in [0, \infty)}$
if necessary, we may assume that $\{ v^*u_t \}_{t \in [0, \infty)}$ is
a norm continuous path of unitaries in $\pi^{-1}(\A_d^{\B})$ such that
$v^*u_0 = 1$.

Now for all $t \in [0,\infty)$, let $\alpha_t \in Aut(\phi(\A) + \B)$ be given
by $\alpha_t(x) =_{df} v^* u_t x u_t^* v$ for all $x \in \phi(\A) + \B$.
Thus, $\{ \alpha_t \}_{t \in [0, \infty}$ is a norm continuous path of
automorphisms of $\phi(\A) + \B$ such that $\alpha_0 = id$.
Hence, by \cite{DadarEilersAsympUE} Proposition 2.15 (see also
\cite{LinStableUniqueness} Theorem 3.2 and 3.4),
there exist a continuous path $\{ v_t \}_{t\in [0, \infty)}$ of unitaries
in $\phi(\A) + \B$ such that
$v_0 = 1$ and
$\| v_t x v_t^* - v^* u_t x u_t^* v \| \rightarrow 0$ as $t \rightarrow \infty$
for all $x \in \phi(\A) + \B$.
Thus, $\| v v_t x v_t^* v^* - u_t x u_t^* \| \rightarrow 0$ as
$t \rightarrow \infty$
for all $x \in \phi(\A) + \B$.

We now proceed as in the last part of the proof of \cite{DadarEilersAsympUE} Proposition 3.6 Step 1 (see
also the proof of \cite{LinStableUniqueness} Theorem 3.4).
For all $t \in [0, \infty)$, let $a_t \in \A$ and $b_t \in \B$ such that
$v v_t = \phi(a_t) + b_t$.
Since $\pi \circ \phi$ is injective, we have that for all $t$,
$a_t$ is a unitary in $\A$, and hence, $\phi(a_t)$ is a unitary in $\phi(\A) + \B$.
Note also that since $\pi \circ \phi = \pi \circ \psi$ and both maps are
injective, $\| a_t a a_t^* - a \| \rightarrow 0$ as $t \rightarrow \infty$ for
all $a \in \A$.
For all $t$, let $w_t =_{df} v v_t \phi(a_t)^* \in 1 + \B$.
Then $\{ w_t \}_{t \in [0,1)}$ is a norm continuous path of unitaries in
$1 + \B$, and for all $a \in \A$,
\begin{eqnarray*}
& & \| w_t \phi(a) w_t^* - \psi(a) \|\\
& \leq & \| w_t \phi(a) w_t^* - v v_t \phi(a) v_t^* v^* \|
+ \| v v_t \phi(a) v_t^* v^* - u_t \phi(a) u_t^* \|
+ \| u_t \phi(a) u_t^* - \psi(a) \|\\
& = & \| v v_t\phi( a_t a a_t^* - a) v_t^* v^* \| + \| v v_t \phi(a) v_t^* v^* - u_t \phi(a) u_t^* \|
+ \| u_t \phi(a) u_t^* - \psi(a) \|\\
& & \rightarrow 0.
\end{eqnarray*}

\end{proof}

We have another generalization of the BDF essential codimension theorem:

\begin{thm}
Let $\A$ be a unital separable simple nuclear C*-algebra, and $\B$ a
separable simple stable C*-algebra with a nonzero projection,
strict comparison of positive elements and for which $T(\B)$ has finitely
many extreme points.

Suppose that there exists a *-embedding $\A \hookrightarrow \B$.

Let $\phi, \psi : \A \rightarrow \Mul(\B)$ be unital extensions
such that
$\phi(a) - \psi(a) \in \B$ for all $a \in \A$.

Then $[\phi, \psi] = 0$ in $KK(\A, \B)$ if and only if
$\phi \underline{\approx} \psi$.
\label{thm:Uniqueness1}
\end{thm}

\begin{proof}
Note that since $\A$ is simple, $\phi$ and $\psi$ are both
norm full extension.
Hence, since $\B$ has the CFP, it follows, by Theorem
\ref{thm:KucerovskyNgAbsorption}, that
$\phi$ and $\psi$ are both unitally absorbing extensions.

The rest of the proof is exactly the same as that of Theorem
\ref{thm:Uniqueness0}, except that Theorem \ref{thm:K1injective} is
replaced with Theorem \ref{thm:K1injectivePart2}.
\end{proof}

We note once more, that, as in Theorem \ref{thm:BDF}, Theorems \ref{thm:Uniqueness0}
and \ref{thm:Uniqueness0} are essentially about local phenomena.

Towards more concrete generalizations, we first need a
technical result.

\begin{lem}
If $\B$ is a nonunital C*-algebra and
$P, Q \in \Mul(\B)$ are projections with $P- Q \in \B$ and $\| P - Q \| < 1$,
then there exists a unitary $U \in 1 + \B$ such that $P = UQU^*$.

Moreover, we can choose
$U$ as above so that $\| U - 1 \| \leq 4 \| P - Q \|$.
\label{lem:CloseKKProj}
\end{lem}
\begin{proof}
Brief sketch of standard argument:  $Z = PQ + (1-P)(1-Q)$ satisfies $Z - 1  =
(1-2P) (P - Q)$,
and thus $\|Z - 1 \|  < 1$.  Hence, $Z$ is invertible and if $U$ is
the unitary in the polar decomposition of $Z$, then $U Q U^* = P$.
Moroever, since $P- Q \in \B$, $Z \in 1 + \B$ and hence, $U \in 1 + \B$.

Also, $\| Z^* Z - 1 \| \leq \|Z^* Z - Z \| + \| Z - 1 \|
\leq \| Z^* - 1 \| \| Z \| + \| Z - 1 \| \leq 3 \| Z - 1\| =
3 \| P - Q \|$.
So $\| |Z| - 1 \| \leq \| |Z|^2 - 1 \| \leq 3 \| P - Q \|$.
So
$\|U - 1 \| \leq \| U  - U |Z| \| + \| Z - 1 \|
= \| 1 - |Z| \| + \| P - Q \| \leq 4 \| P - Q \|$.

\end{proof}

We now move towards a more concrete generalization of the BDF essential codimension
theorem.  We will be using the notion of generalized essential codimension
in Definition \ref{df:generalizedessentialcodimension}.

\begin{thm}
Let $\B$ be a separable stable simple purely infinite
 C*-algebra, and $P, Q \in \Mul(\B)$ projections
such that $P, Q, 1-P, 1-Q \notin \B$,  and
$P - Q \in \B$.

Then $[P:Q] = 0$ in $K_0(\B)$
if and only if there exists a unitary $U \in 1 + \B$ such that
$U P U^* = Q$.
\label{thm:BDFQuarterway}
\end{thm}

\begin{proof}
Since $\B$ is simple purely infinite, it follows, from the hypotheses,
that $P \sim 1 - P \sim Q \sim 1 - Q \sim 1$.
Let $\phi, \psi : \mathbb{C} \rightarrow \Mul(\B)$ be *-homomorphisms such that
$\phi(1) = P$ and $\psi(1) = Q$. Then
$\phi$ and $\psi$ are absorbing trivial extensions.
(And both absorb the zero extension.)

The ``if'' direction then follows immediately from Theorem \ref{thm:Uniqueness0}.
(See also \cite{LeeJFA2013} Lemma 2.4.)

We now prove the ``only if'' direction.
We have that $[\phi, \psi] = [P:Q] = 0$.
Hence, by Theorem \ref{thm:Uniqueness0}, there exists a norm continuous path
$\{ u_t \}_{t \in [0,1]}$ of unitaries in $\mathbb{C}1 + \B$ such that
$\| u_t P u_t^* - Q \| \rightarrow 0$ as $t \rightarrow \infty$.

Choose $s \in [0, \infty)$ such that
$\| u_s P u_s^* - Q \| < 1$.
We may assume that $u_s \in 1 + \B$.
Then, by Lemma \ref{lem:CloseKKProj}, there exists
a unitary $V \in 1 + \B$ such that
$V u_s P u_s^* V^* = Q$.
Take $U =_{df} V u_s$.
\end{proof}

We note that there is a mistake in \cite{LeeFirst}
Theorem 2.14.  It is \emph{not true} that if $\B$ is a separable simple
stable purely infinite C*-algebra for which $\Mul(\B)$ has real
rank zero, and if $P, Q \in \Mul(\B)$ are projections
with $P - Q \in \B$, $P \notin \B$, for which $[P, Q] = 0$ in
$K_0(\B)$ then there exists a unitary $U \in 1 + \B$ such that
$U P U^* = Q$.  Here is a counterexample:  Take $\B = O_2 \otimes \K$
and let $r \in O_2 \otimes \K$ be a nonzero projection.
Note that $[r] = 0$ in $K_0(O_2)$.  Let
$P =_{df} 1_{\Mul(O_2 \otimes \K)}$ and $Q = P - r$.
Then $P - Q = r \in O_2 \otimes \K$, $P \notin O_2 \otimes \K$,
and $[P:Q] = 0$ in $K_0(O_2 \otimes \K)$.  But it is \emph{not true}
that $P$ is unitarily equivalent to $Q$.

The mistake in the argument of \cite{LeeFirst} Theorem 2.14
is
essentially a mistake about absorbing extensions.
If $\phi : \A \rightarrow \Mul(\B)$ is an absorbing extension
then $\phi \oplus 0 \sim \phi$, i.e., $\phi$ must absorb the $0$
extension, and thus $ran(\phi)^{\perp}$ must be big.
(Of course, this must be separated from the unital case
where $\phi(1) = 1$ and $\phi$ is \emph{unitally absorbing} --
meaning absorbing all \emph{strongly unital} trivial extensions.)

Finally, we note that in a separate paper, where we also investigate
the relationship between essential codimension and projection lifting,
we will look more extensively at concrete generalizations of the
BDF essential codimension result, as in the above.

\section{Technical lemma}
\label{sec:TechnicalLemma}

For $\delta > 0$,
let $f_{\delta} : [0, \infty) \rightarrow [0,1]$ be the unique continuous function for which
\[
f_{\delta}(t)
=
\begin{cases}
1 & t \in [\delta, \infty) \\
0 & t = 0\\
\makebox{linear on  } & [0, \delta].
\end{cases}
\]

If $\C$ is a unital C*-algebra and $p \in \C$ is a projection, we follow standard
convention by letting $p^{\perp} =_{df} 1 - p$.

In what follows, for elements $a,b$ in a C*-algebra, we use $a \approx_{\epsilon} b$ to denote $\Vert a-b \Vert < \epsilon$.





\begin{lem}
Let $\B$ be a separable stable C*-algebra with an approximate
unit $\{ e_n \}$ consisting of increasing projections.
(We define $e_0 =_{df} 0$.)

Suppose that $A, A', A'' \in \C(\B)_+$ are contractive elements and $\delta > 0$
such that
$$A A' = A'$$
and
$$A'' \in her((A' - \delta)_+).$$

Let $A_0 \in \Mul(\B)$ be any contractive lift of $A$, and let $\epsilon > 0$ be
given.

Then for every $M \geq 0$, there exists an $A''_0 \in
e_M^{\perp}\Mul(\B) e_M^{\perp}$ which is a contractive positive
lift of $A''$ such that for all $l \geq 1$,
$$A_0 (A''_0)^{1/l} \approx_{\epsilon} A''_0 \approx_{\epsilon} (A''_0)^{1/l} A_0.$$
\label{lem:Oct1920186AM}
\end{lem}

\begin{proof}
Choose $\delta_1 > 0$ such that for any contractive operators $B, C$, with
$C \geq 0$, if
$$BC \approx_{\delta_1} C \approx_{\delta_1} CB$$
then
$$B f_{\delta}(C) \approx_{\epsilon} f_{\delta}(C) \approx_{\epsilon}
f_{\delta}(C) B.$$
(Sketch of argument for choosing $\delta_1$:  By the Weierstrass approximation theorem,
find a polynomial $p(t)$, with $p(0) = 0$, such that $|f_{\delta}(t) - p(t) | < \epsilon/2$ for all
$t \in [0,1]$. Now use the concrete structure of $p(t)$ to determine $\delta_1$.)

Let $A'_0 \in e_M^{\perp}\Mul(\B)e_M^{\perp}$ be any contractive positive lift of $A'$.
Note that we can restrict to the corner $e_M^{\perp}$ because the image of $\pi(e_M^{\perp}) = 1$ since $e_M \in \Mul(\B)$.
Because $AA' = A'$ (and since they are positive, we also have $A' = A'A$) it follows that $AA'^{1/2} = A'^{1/2}$ and so also $A'^{1/2}A = A'^{1/2}$.
Therefore, there exist $c, c' \in \B$ for which
$$A_0 {A'_0}^{1/2} = {A'_0}^{1/2} + c,$$
and
$${A'_0}^{1/2} A_0 = {A'_0}^{1/2} + c'.$$
Since $\{e_n \}$ is an approximate identity for $\B$, we can choose $N \geq 1$ so that
$$c e_N^{\perp} \approx_{\delta_1} 0 \approx_{\delta_1} e_N^{\perp} c'.$$
Then, combining the above displays yields
$$A_0 {A'_0}^{1/2} e_N^{\perp} \approx_{\delta_1} {A'_0}^{1/2} e_N^{\perp},$$
and
$$e_N^{\perp} {A'_0}^{1/2} A_0 \approx_{\delta_1} e_N^{\perp} {A'_0}^{1/2}.$$

Hence,
if we define
$$D =_{df} {A'_0}^{1/2} e_N^{\perp}{A'_0}^{1/2}$$
then
$$A_0 D \approx_{\delta_1} D \approx_{\delta_1} D A_0.$$

Hence, by the definition of $\delta_1$,
\begin{equation}
A_0 f_{\delta}(D) \approx_{\epsilon} f_{\delta}(D)
\approx_{\epsilon}  f_{\delta}(D) A_0.\label{eq:fdeltaD-almost-comm-A0}
\end{equation}

Note that $\pi(D) = A'$, which follows since $\pi(e_N^{\perp}) = 1$.
Because the algebra $\pi(\overline{(D - \delta)_+ \Mul(\B) (D - \delta)_+}) = her((A' - \delta)_+)$, we can find
a contractive positive lift $A''_0 \in \overline{(D - \delta)_+ \Mul(\B) (D - \delta)_+}$
of $A''$.
Note that $A''_0 \in e_M^{\perp} \Mul(\B) e_M^{\perp}$ since $A'_0$, and consequently, $D$ and $(D-\delta)_+$ are.

We remark that $f_{\delta}(D) (D-\delta)_+ = (D-\delta)_+$, and for all $l \geq 1$, $A''^{1/l}_0$ is a contraction.
Combining these facts with \eqref{eq:fdeltaD-almost-comm-A0} we obtain
$$A_0 (A''_0)^{1/l} = A_0 f_{\delta}(D) (A''_0)^{1/l}
\approx_{\epsilon} f_{\delta}(D) (A''_0)^{1/l} = (A''_0)^{1/l}.$$
Similarly,
\[(A''_0)^{1/l} A_0 \approx_{\epsilon} (A''_0)^{1/l}. \qedhere\]
\end{proof}

\vspace*{2ex}

\textbf{We now
fix some notation which will be used for the rest of this section.}

\vspace*{1ex}

Let $\B$ be a separable simple stable C*-algebra with a nonzero projection.
Let $\{ p_k \}_{k=1}^{\infty}$ be a sequence of pairwise orthogonal
projections of $\B$
such that
$$\sum_{k=1}^{\infty} p_k = 1_{\Mul(\B)}.$$
where the series converges strictly.

For all $m \leq n$, let
$$p_{m, n}  =_{df} \sum_{k=m}^n p_n$$
and
let
$$e_n =_{df} \sum_{k=1}^n p_n.$$
(Hence, $\{ e_n \}$ is an approximate unit for $\B$.)

Let $U \in \C(\B)$ be a unitary and
let $V \in \Mul(\B)$ be a partial isometry such that
$$\pi(V) = U.$$

Also, we let $\overline{B(0,1)}$ denote the closed unit ball of the
complex plane, i.e., $\overline{B(0,1)} =_{df} \{ \alpha \in
\mathbb{C} : |\alpha| \leq 1 \}$.

Recall also that for a  C*-algebra $\C$, for a
$\tau \in T(\C)$ and for any $a \in \C_+$,
$$d_{\tau}(a) =_{df} \lim_{n\rightarrow \infty} \tau(a^{1/n}).$$

\begin{lem}

Let $h_1, h_{2}, h_{3}: S^1 \rightarrow [0,1]$ be continuous functions and
let $\delta_1 > 0$ be such that
$$h_1 h_{2} = h_{2}$$
and
$$\overline{\supp(h_{3})} \subset \supp((h_{2} - \delta_1)_+).$$

Let $\delta_2 > 0$ and $\widehat{h}$ be a complex polynomial such that
$$|\widehat{h}(\lambda) - h_1(\lambda) | < \frac{\delta_2}{10}$$
for all $\lambda \in S^1$.

Then for every $L, L'  \geq 1$, there exist  $L_1 > L'$,
there exist contractive $A \in e_{L_1}^{\perp} \Mul(\B)_+ e_{L_1}^{\perp}$ which is a lift of
$h_{3}(U)$ such that
for every contractive $a \in (\overline{A \B A})_+$
there exist $M > L$, $M_1 > L_1$ and
$x \in  p_{L_1 + 1, M_1}\B p_{L_1 +1, M_1}$ for which
$$x \widehat{h} \left( \sum_{j=1}^{\infty} \alpha_j p_j V \right) x^*
\approx_{\delta_2} a$$
for every sequence $\{ \alpha_j \}$  in $\overline{B(0,1)}$ (closed unit
ball of the complex plane) such
that $\alpha_j = 1$ for all $L \leq j \leq M$.
\label{lem:Oct2020186AM}
\end{lem}

\begin{proof}

Let $A_0 \in \Mul(\B)_+$ be a contractive lift of $h_1(U)$.

Since $U$ is unitary and because of the conditions on $\hat h$, we know $\hat h(U) \approx_{\frac{\delta_2}{10}} h_1(U)$.
Moreover, since $\hat h$ is a polynomial, $\hat h(U) = \hat h(\pi(V)) = \pi(\hat h(V))$ and also $h_1(V) = \pi(A_0)$.
Using these facts along with the fact that $\{e_n \}$ is an approximate identity for $\B$, we can
choose $L_1 > L'$ so that

\begin{equation}
e_{L_1}^{\perp} \widehat{h}(V) e_{L_1}^{\perp} \approx_{\frac{\delta_2}{10}}
e_{L_1}^{\perp} A_0 e_{L_1}^{\perp}\label{eq:hathV-approx-A0}
\end{equation}

and

$$\widehat{h}\left(\sum_{j=1}^{\infty} \alpha_j p_j V \right)
e_{L_1}^{\perp} \approx_{\frac{\delta_2}{10}}
\widehat{h}\left(\sum_{j=1}^{\infty} \alpha'_j p_j V \right) e_{L_1}^{\perp}$$

for all sequences $\{ \alpha_j \}$ and $\{ \alpha'_j \}$
in $\overline{B(0,1)}$ such that $\alpha_j = \alpha'_j$ for all $j \geq L$.

By Lemma \ref{lem:Oct1920186AM} (instantiated with $A, A', A'', \delta, \epsilon, M$ chosen to be $h_1(U), h_2(U),$ $h_3(U), \delta_1, \frac{\delta_2}{10}, L_1$), there exists
$A \in e_{L_1}^{\perp}\Mul(\B) e_{L_1}^{\perp}$ which is a contractive positive lift of
$h_{3}(U)$ for which
$$e_{L_1}^{\perp} A_0 e_{L_1}^{\perp} A^{1/l} \approx_{\frac{\delta_2}{10}}
A^{1/l}$$
for all $l \geq 1$.

Hence, if we let $a \in \overline{A \B A}$ be an arbitrary contractive positive
element, then because the previous display holds for all $l \ge 1$,
$$e_{L_1}^{\perp} A_0 e_{L_1}^{\perp} a^{1/2}
\approx_{\frac{\delta_2}{10}} a^{1/2}.$$
Chaining this with \eqref{eq:hathV-approx-A0} yields
$$e_{L_1}^{\perp} \widehat{h}(V) e_{L_1}^{\perp} a^{1/2}
\approx_{\frac{\delta_2}{5}} a^{1/2}.$$
Therefore,
if we let $y =_{df} a^{1/2}$ then
$$y e_{L_1}^{\perp} \widehat{h}(V) e_{L_1}^{\perp} y^*
\approx_{\frac{2\delta_2}{5}} a.$$

By the definition of $L_1$ and since $y, a \in \B$,
we can choose $M > L$ and $M_1 > L_1$ such that if
we define
$$x =_{df} p_{L_1 +1, M_1}y p_{L_1 +1, M_1}$$
then
$$x \widehat{h}(\sum_{j=1}^{\infty} \alpha_j p_j V) x^*  \approx_{\delta_2}
a$$
for every sequence $\{ \alpha_j \}$ in $\overline{B(0,1)}$ for which
$\alpha_j = 1$ for all $L \leq j \leq M$.
\end{proof}

\begin{lem}
Suppose that, in addition, $\B$ has strict comparison of positive elements
and $T(\B)$ has  finitely many extreme points.

Let $p \in \B$ be a nonzero projection and let $\epsilon > 0$ be given.

Let $h_1, h_2, h_3 : S^1 \rightarrow [0,1]$ be a continuous functions $\delta_1 > 0$ and
$\lambda_1, ..., \lambda_m \in S^1$
such that
$$h_1 h_2 = h_2$$
$$\overline{\supp(h_3)} \subset \supp((h_2 - \delta_1)_+)$$
and
the function
$$\lambda \mapsto \sum_{j=1}^m h_3(\lambda_j \lambda)$$
is a full element in $C(S^1)$.

There exists $\delta_2 > 0$ such that if
$\widehat{h}$ a complex polynomial for which
$$|\widehat{h}(\lambda) - h_1(\lambda)| < \frac{\delta_2}{10}$$
for all $\lambda \in S^1$ then the following holds:

For every $L, L' \geq 1$,  there exist
$L < L_1 < L_2 < ... < L_m$,  $L' < M < M'$,
a projection $q \in \B$ for which $q \sim p$,
and
contractive $x \in p_{M+1, M'}\B p_{M +1, M'}$ such that
$$x \widehat{h}\left(\sum_{j=1}^{\infty} \alpha_j p_j V\right) x^*
\approx_{\epsilon} q$$
where $\{ \alpha_j \}$ is any sequence in $\overline{B(0,1)}$ such that
$\alpha_j = \lambda_k$ for all $L_{k-1} < j \leq L_k$ and all $1 \leq k \leq m$.
(Here, $L_0 =_{df} L$.)
\label{lem:Oct2020187AM}
\end{lem}

\begin{proof}
Let $\F$ be the finitely many extreme points of $T(\B)$.

Let $L, L' \geq 1$ be arbitrary.

Let $\epsilon > 0$ and let $\delta_2 > 0$ be any constant for which $\delta_2 < \frac{\epsilon}{m^2}$.

We construct a elements
$L_j$, $A_j$, $b_j$, $\epsilon_j$, $M$, $M'_j$, $M''_j$, $M'''_j$
and $x_j$ for $1 \leq j \leq m$.
The construction is by induction on $j$.\\

\noindent \emph{Basis step:} $j = 1$.

By Lemma \ref{lem:Oct2020186AM}, choose
$M > L'$ and contractive positive
$A_1 \in e_{M}^{\perp} \Mul(\B) e_{M}^{\perp}$  such that
$$\pi(A_1) = h_3(\lambda_1 U).$$

We let $$M'_1 =_{df} M.$$

Let $\F_1 =_{df} \{ \tau \in \F : \tau(A_1) = \infty \}$.

Let $a_1 \in \overline{A_1 \B A_1}$ be a strictly positive element.
Choose
$\epsilon_1 >0$  so that
$$d_{\tau}(a_1 - 2 \epsilon_1)_+ >
\tau(p)$$
for all $\tau \in \F_1$.

By Lemma \ref{lem:Oct2020186AM}, choose $L_1 > L$, $M''_1 > M'_1 = M$ and
a contractive element $x_1 \in p_{M+1, M''_1} \B p_{M+1, M''_1}$ so that
$$x_1 \widehat{h}\left(\sum_{j=1}^{\infty} \alpha_j p_j V \right) x_1^*
\approx_{\delta_2} f_{\epsilon_1}(a_1)$$
for every sequence $\{ \alpha_j \}$ in $\overline{B(0,1)}$
for which $\alpha_j = \lambda_1$ for all $L < j \leq L_1$.

We let $M'''_1 > M''_1$ be a number that is big enough so that
if we define
 $$b_1 =_{df} p_{M+1, M'''_1} a_1 p_{M+1, M'''_1}$$
then
$$d_{\tau}((b_1 - 2 \epsilon_1)_+) > \tau(p)$$
for all $\tau \in \F_1$ and
$$x_1 \widehat{h}\left(\sum_{j=1}^{\infty} \alpha_j p_j V \right) x_1^*
\approx_{\delta_2} f_{\epsilon_1}(b_1)$$
for every sequence $\{ \alpha_j \}$ in $\overline{B(0,1)}$
for which $\alpha_j = \lambda_1$ for all $L < j \leq L_1$.\\

\noindent \emph{Induction step.}
Suppose that $L_k$, $A_k$, $b_k$, $\epsilon_k$, $M'_k$, $M''_k$
$M'''_k$, $\epsilon_k$ and $x_k$ have been chosen.
We now construct the constants with $k$ replaced with $k + 1$.

Choose $N > M'''_k$ big enough so that
$$\left\| e_{M'''_k}
\widehat{h}(\sum_{j=1}^{\infty} \alpha_j p_j V) e_N^{\perp}
\right\|,
\left\| e_N^{\perp} \widehat{h}(\sum_{j=1}^{\infty} \alpha_j p_j V)
e_{M'''_k} \right\| < \delta_2,$$
for every sequence $\{ \alpha_j \}$ in $\overline{B(0,1)}$.

By Lemma \ref{lem:Oct2020186AM}, choose
$M'_{k+1} > N$
and a contractive positive element
$A_{k+1} \in e_{M'_{k+1}}^{\perp} \Mul(\B) e_{M'_{k+1}}^{\perp}$
such that
$$\pi(A_{k+1}) = h_3(\lambda_{k+1} U).$$

Let $\F_{k+1} =_{df} \{ \tau \in \F : \tau(A_{k+1}) = \infty \}.$
Let $a_{k+1} \in \overline{A_{k+1} \B A_{k+1}}$ be a strictly positive
element.  Choose $\epsilon_{k+1} > 0$ so that
$$d_{\tau}((a_{k+1} - 2 \epsilon_{k+1})_+) > \tau(p)$$
for all $\tau \in \F_{k+1}$.

By Lemma \ref{lem:Oct2020186AM},
choose $L_{k+1} > L_k$, $M''_{k+1} > M'_{k+1}$ and
contractive
$x_{k+1} \in p_{M'_{k+1} +1, M''_{k+1}} \B p_{M'_{k+1} +1, M''_{k+1}}$
so that
$$x_{k+1} \widehat{h}\left( \sum_{j=1}^{\infty} \alpha_j p_j V \right)
x_{k+1}^* \approx_{\delta_2} f_{\epsilon_{k+1}}(a_{k+1})$$
for every sequence $\{ \alpha_j \}$ in $\overline{B(0, 1)}$ for which
$\alpha_j = \lambda_{k+1}$ for all $L_k < j \leq L_{k+1}$.

Find $M'''_{k+1} > M''_{k+1}$ big enough so that if
we define
$$b_{k+1} =_{df} p_{M'_{k+1}, M'''_{k+1}} a_{k+1} p_{M'_{k+1},
M'''_{k+1}}$$
then
$$d_{\tau}((b_{k+1} - 2 \epsilon_{k+1})_+) > \tau(p)$$
for all $\tau \in \F_{k+1}$.
and
$$x_{k+1} \widehat{h}\left( \sum_{j=1}^{\infty} \alpha_j p_j V \right)
x_{k+1}^* \approx_{\delta_2} f_{\epsilon_{k+1}}(b_{k+1})$$
for every sequence $\{ \alpha_j \}$ in $\overline{B(0, 1)}$ for which
$\alpha_j = \lambda_{k+1}$ for all $L_k < j \leq L_{k+1}$.\\

This completes the inductive construction.\\

Now let $x \in p_{M+1, M'''_{m+1}} \B p_{M+1, M'''_{m+1}}$
be the contractive element defined by
$$x =_{df} \sum_{j=1}^m x_j.$$

Let $\{ \alpha_j \}$ be any sequence in $\overline{B(0,1)}$ such that
$\alpha_j = \lambda_k$ for all $L_{k-1} < \alpha_j \leq L_k$ and for all
$1 \leq k \leq m$.  (Here $L_0 =_{df} L$.)

Then
\begin{eqnarray*}
& & x \widehat{h}(\sum_{j=1}^{\infty} \alpha_j p_j V) x^* \\
& \approx_{(m^2 - m) \delta_2} & \sum_{k=1}^m
x_k \widehat{h}(\sum_{j=1}^{\infty} \alpha_j p_j V) x_k^*\\
& \approx_{m \delta_2} & \sum_{k=1}^m f_{\epsilon_k}(b_k).\\
\end{eqnarray*}

Since
$$d_{\tau}\left(\sum_{k=1}^m (b_k - 2 \epsilon_k)_+ \right) >
\tau(p)$$
for all $\tau \in T(\B)$ and since $\B$ has strict comparison for
positive elements, there exists
a projection $q \in \overline{(\sum_{k=1}^m (b_k - 2 \epsilon_k)_+) \B
(\sum_{k=1}^m (b_k - 2\epsilon_k)_+)}$ such that
$$q \sim p.$$

Hence,
$$qx \widehat{h}\left( \sum_{j=1}^{\infty} \alpha_j p_j V \right) x^* q
\approx_{m^2 \delta_2} q$$
for every sequence $\{ \alpha_j \}$ in $\overline{B(0,1)}$ for which
$\alpha_j = \lambda_k$ for every $L_{k-1} <  j \leq L_{k}$ and
all $1 \leq k \leq m$.
Since $\delta_2 < \frac{\epsilon}{m^2}$ we are  done.
\end{proof}

Let $\D$ be a unital C*-algebra. Recall that a nonzero element $x \in \D$
is said to be \emph{strongly full} in $\D$ if every nonzero element
of $C^*(x)$ is a full element of $\D$.

\begin{lem}
Say that, in addition,
$\B$ has strict comparison for positive elements and $T(\B)$ has
finitely many extreme points.

Then there exists a sequence $\{ \alpha_j \}$ in $S^1$ such that the unitary
$\pi(\sum_{j=1}^{\infty} \alpha_j p_j)U$ is a strongly full element
of $\C(\B)$.
\label{lem:TechnicalLemma}
\end{lem}

\begin{proof}

For every $k \geq 1$,
let $h_{k, 1, j}, h_{k, 2, j}, h_{k, 3, j} : S^1 \rightarrow [0,1]$
be continuous functions, $\lambda_{k,j} \in S^1$ (for $1 \leq j \leq k$),
and
$\delta_k > 0$ be such that
$$\sum_{j=1}^k h_{k, 3, j}$$
is a full element of  $C(S^1)$,
$$\overline{\supp(h_{k, 3, j})} \subset \supp((h_{k, 2, j} - \delta_k)_+),$$
$$h_{k, 1, j} h_{k, 2, j} = h_{k, 2, j},$$
$$h_{k, 3, j}(\lambda) = h_{k, 3, 1}(\lambda_{k,j} \lambda)$$
for all $1 \leq j \leq k$, and
$$\max_{1 \leq j \leq k} \diam(\supp(h_{k, 1, j})) \rightarrow 0$$
as $k \rightarrow \infty$.

Let $\{ h_l \}_{l=1}^{\infty}$ be a sequence of continuous functions
from $S^1$ to $[0,1]$ such that for all $l$, there exists $k, j$
such that $h_l = h_{k, 1, j}$
and for all $k, j$, $h_{k, 1,j}$ occurs infinitely many times as a term
in the sequence $\{ h_l \}_{l=1}^{\infty}$.

Let $\{ r_n \}$ be a sequence of pairwise orthogonal projections in
$\B$ such
$r_m \sim r_n$ for all $m, n$, and
$$\sum_{n=1}^{\infty} r_n = 1_{\Mul(\B)}$$
where the sum converges strictly.

Let $\{ \epsilon_k \}$ be a decreasing sequence in $(0,1)$ and
$\{ \epsilon_{k,l} \}$ a (decreasing in $k+l$)
biinfinite sequence in $(0,1)$ such that
$$\sum_{k=1}^{\infty} \epsilon_k < \infty$$
and
$$\sum_{1 \leq k, l < \infty} \epsilon_{k,l} < \infty.$$

Note that for every $\gamma > 0$, $Y \in \Mul(\B)$ and $y \in \B$,
there exists $N \geq 1$ such that
$$\| y Y e_N^{\perp} \| < \gamma.$$
Also, for every $\gamma > 0$,
complex polynomial $\widehat{h}$ and $L \geq 1$, there exists $N \geq 1$
so that
$$\widehat{h}\left( \sum_{j=1}^{\infty} \beta_j p_n V \right) e_N^{\perp}
\approx_{\gamma} \widehat{h}\left( \sum_{j=1}^{\infty} \beta'_j p_n V
\right)
e_N^{\perp}$$
for all sequences $\{ \beta_j \}$ and $\{ \beta'_j \}$ in
$\overline{B(0,1)}$ for which
$\beta_j = \beta'_j$ for all $j \geq L$.

By using the above two principles and by repeatedly applying
Lemma \ref{lem:Oct2020187AM} (first to $h_1$;  then to $h_2$; then to
$h_3$; and so forth),
we can find a sequence $\{ x_k \}$ of pairwise orthogonal contractive
elements of $\B$, a sequence $\{ \alpha_k \}$ in $S^1$, and a sequence
$\{ \widehat{h}_k \}$ of complex polynomials such that the
following statements hold:

\begin{enumerate}
\item $\sum_{k=1}^{\infty} x_k$ converges strictly in $\Mul(\B)$.
\item For all $k \geq 1$,
$$\max_{\lambda \in S^1} | h_k(\lambda) - \widehat{h}_k(\lambda) | <
\epsilon_k.$$
\item \label{Oct2020188AM}
 For all $k \geq 1$, there exists a subsequence $\{ x_{j_l} \}$
of $\{ x_j \}$ such that
\begin{enumerate}
\item $\left\| x_{j_l} \widehat{h}_k\left(\sum_{n=1}^{\infty} \alpha_n
p_n V\right) x_{j_s}^* \right\| < \epsilon_{l, s}$ for all
$l \neq s$, and
\item $x_{j_l} \widehat{h}_k\left(\sum_{n=1}^{\infty} \alpha_n
p_n V\right) x_{j_l}^* \approx_{\epsilon_{j_l}} r_{j_l}$  for all $l$.
\end{enumerate}
\end{enumerate}
We denote the above statements by ``(*)".

(Sketch of argument on how to choose the subsequence in (*) (3) above:
Firstly, from the construction of the sequence, we already have part (3)(b).
Next, note that, from Lemma \ref{lem:Oct2020187AM}, there is a sequence of pairwise orthogonal
projections $\{ s_j \}$ in $\B$ such that $\sum_{j=1}^{\infty} s_j$ converges strictly in $\Mul(\B)$ and
$x_j = r_j x_j s_j$ for all $j$.
Now fix a $k$.  The subsequence is  constructed in two steps (a
subsequence of a subsequence).  Step 1: Let $\{ j_i \}$ be a subsequence of the positive integers
for which $h_{j_i} = h_k$ for all $i$.  Step 2:  Extract the subsequence of
$\{ j_i \}$ by observing that for all $\delta$, for all $Y \in \Mul(\B)$, for all
$i_1$, there exists $i_2$ such that for all $i \geq i_2$,
$\| x_{j_{i_1}} Y x_{j_i}^* \| < \delta$.)

Let $m \geq 1$ be given.
We will now show that $h_m(\pi(\sum_{n=1}^{\infty} \alpha_n p_n) U)$ is full
in $\C(\B)$.
Let $\epsilon > 0$.
Since each term of the sequence $\{h_l\}_{l=1}^{\infty}$ is repeated infinitely many times there is some $k$ for which $\epsilon_k < \frac{\epsilon}{2}$ and $h_k = h_m$.

Choose a subsequence $\{ x_{j_l} \}$
of $\{ x_j \}$ as in (\ref{Oct2020188AM}) of (*), corresponding to $\hat h_k$.
Let $A \in \Mul(\B)_+$ be a contractive element so that
$$\pi(A) = h_k(\pi(\sum_{n=1}^{\infty} \alpha_n p_n) U).$$

We can choose $L \geq 1$ great enough so that if we define
$$X =_{df} \sum_{l=L}^{\infty} x_{j_l}$$
then
$$X \widehat{h}_k\left( \sum_{n=1}^{\infty} \alpha_n p_n V \right) X^*
\approx_{\epsilon_k} X A X^*.$$

Increasing $L$ if necessary, we may assume that
$$\sum_{l \geq L} \epsilon_l + \sum_{m, n \geq L} \epsilon_{m, n}
< \frac{\epsilon}{2}.$$

Consider the projection $R =_{df} \sum_{l \ge L} r_{j_l} \in \Mul(\B)$ and note that $R \sim 1_{\Mul(\B)}$ since $\sum_{n=1}^{\infty} r_n \sim 1,$ and because all the projections $r_n$ are equivalent.
From (\ref{Oct2020188AM}) of (*),
$$X \widehat{h}_k\left( \sum_{n=1}^{\infty} \alpha_n p_n V \right) X^*
\approx_{\delta} R $$
where
$$\delta =_{df} \sum_{l=L}^{\infty} \epsilon_{j_l} +
\sum_{L \leq l, s < \infty} \epsilon_{l, s}.$$

Therefore,
$$\| X A X^* - R \| < \delta + \epsilon_k < \epsilon.$$
Since $R \sim 1_{\Mul(\B)}$, there is some partial isometry $W$ implementing the equivalence so that $WW^{*} = 1_{\Mul(\B)}$ and $W^{*}W = R$.
Then
\begin{equation*}
  \| WXAX^{*}W^{*} - 1_{\Mul(\B)} \| \le \| XAX^{*} - R \| < \epsilon.
\end{equation*}

Applying $\pi$, we obtain
\begin{equation*}
  \| \pi(WX)\pi(A)\pi(X^{*}W^{*}) - 1_{\C(\B)} \| < \epsilon
\end{equation*}
Therefore, $\pi(A) = h_k(\pi(\sum_{n=1}^{\infty} \alpha_n p_n) U) = h_m(\pi(\sum_{n=1}^{\infty} \alpha_n p_n) U)$ is
full in $\C(\B)$.

Since $m \ge 1$ was arbitrary, and by the definition of the sequence
$\{ h_k \}$, we claim that $\pi(\sum_{n=1}^{\infty} \alpha_n p_n)U$
is a strongly full element of $\C(\B)$.

To see this, note that every nonnegative continuous function $f \in C(S^1)$ has some $h_l$ which is in the ideal generated by $f$.
Indeed, there is some arc of positive width $\eta$ centered at $s \in S^1$ on which $f$ is greater than some $\zeta > 0$.
Since $\max_{1 \le j \le k} \diam (\supp (h_{k,1,j})) \to 0$, there is some $k$ such that the maximum of these diameters is less than $\frac{\eta}{3}$.
Moreover, since $\sum_{j=1}^k h_{k,3,j}$ is full in $C(S^1)$, there is some $1 \le j \le k$ such that $h_{k,1,j}(s) \not= 0$.
Then, because $\diam(\supp(h_{k,1,j})) < \frac{\eta}{3}$, the support of $h_{k,1,j}$ is entirely contained within the arc on which $f \ge \zeta > 0$.
Therefore $h_{k,1,j}$ is in the ideal generated by $f$.
Finally, by the definition of $\{ h_l \}_{l=1}^{\infty}$, there is some $l$ for which $h_l = h_{k,1,j}$ (in fact, there are infinitely many such $l$).

\end{proof}

\bibliographystyle{plain}
\bibliography{references}

\begin{thebibliography}{10}

\bibitem{BCP+-2006-Agatoeo}
Moulay-Tahar Benameur, Alan~L. Carey, John Phillips, Adam Rennie, Fyodor~A.
  Sukochev, and Krzysztof~P. Wojciechowski.
\newblock An analytic approach to spectral flow in von {N}eumann algebras.
\newblock In Matthias Lesch, Bernhelm Boo-Bavnbek, Slawomir Klimek, and Weiping
  Zhang, editors, {\em Analysis, geometry and topology of elliptic operators},
  pages 297--352. World Sci. Publ., Hackensack, NJ, 2006.

\bibitem{Blackadar}
Bruce Blackadar.
\newblock {\em {$K$}-theory for operator algebras}, volume~5 of {\em
  Mathematical Sciences Research Institute Publications}.
\newblock Cambridge University Press, Cambridge, second edition, 1998.

\bibitem{BlanchardRohdeRordam}
Etienne Blanchard, Randi Rohde, and Mikael R\o{r}dam.
\newblock Properly infinite {$C(X)$}-algebras and {$K_1$}-injectivity.
\newblock {\em J. Noncommut. Geom.}, 2(3):263--282, 2008.

\bibitem{BJ-2015-TAMS}
Marcin Bownik and John Jasper.
\newblock The {S}chur-{H}orn theorem for operators with finite spectrum.
\newblock {\em Trans. Amer. Math. Soc.}, 367(7):5099--5140, 2015.

\bibitem{BDF1}
Lawrence~G. Brown, Ronald~George Douglas, and Peter~Arthur Fillmore.
\newblock Unitary equivalence modulo the compact operators and extensions of
  {$C^{\ast} $}-algebras.
\newblock In Peter~Arthur Fillmore, editor, {\em {Proceedings of a Conference
  on Operator Theory}}, volume 345 of {\em Lecture Notes in Mathematics}, pages
  58--128. Springer, Berlin, 1973.

\bibitem{BrownLee}
Lawrence~G. Brown and Hyun~Ho Lee.
\newblock Homotopy classification of projections in the corona algebra of a
  non-simple {$C^*$}-algebra.
\newblock {\em Canad. J. Math.}, 64(4):755--777, 2012.

\bibitem{CPS-2003-AM}
Alan Carey, John Phillips, and Fyodor~A. Sukochev.
\newblock Spectral flow and {D}ixmier traces.
\newblock {\em Adv. Math.}, 173(1):68--113, 2003.

\bibitem{DadarEilersAsympUE}
Marius Dadarlat and S\o{r}en Eilers.
\newblock Asymptotic unitary equivalence in {$KK$}-theory.
\newblock {\em $K$-Theory}, 23(4):305--322, 2001.

\bibitem{ElliottKucerovsky}
George~A. Elliott and Dan Kucerovsky.
\newblock An abstract {V}oiculescu-{B}rown-{D}ouglas-{F}illmore absorption
  theorem.
\newblock {\em Pacific J. Math.}, 198(2):385--409, 2001.

\bibitem{HigsonKK}
Nigel Higson.
\newblock A characterization of {$KK$}-theory.
\newblock {\em Pacific J. Math.}, 126(2):253--276, 1987.

\bibitem{HigsonPaschkeDual}
Nigel Higson.
\newblock {$C^*$}-algebra extension theory and duality.
\newblock {\em J. Funct. Anal.}, 129(2):349--363, 1995.

\bibitem{Jas-2013-JFA}
John Jasper.
\newblock {The Schur--Horn theorem for operators with three point spectrum}.
\newblock {\em J. Funct. Anal.}, 265(8):1494--1521, 2013.

\bibitem{JensenThomsenBook}
Kjeld~Knudsen Jensen and Klaus Thomsen.
\newblock {\em Elements of {$KK$}-theory}.
\newblock Mathematics: Theory \& Applications. Birkh\"auser Boston, Inc.,
  Boston, MA, 1991.

\bibitem{Kad-2002-PNASU}
Richard~V. Kadison.
\newblock {The Pythagorean Theorem I: the finite case}.
\newblock {\em Proc. Natl. Acad. Sci. USA}, 99(7):4178--4184, 2002.

\bibitem{Kad-2002-PNASUa}
Richard~V. Kadison.
\newblock {The Pythagorean Theorem II: the infinite discrete case}.
\newblock {\em Proc. Natl. Acad. Sci. USA}, 99(8):5217--5222, 2002.

\bibitem{KL-2017-IEOT}
Victor Kaftal and Jireh Loreaux.
\newblock Kadison's pythagorean theorem and essential codimension.
\newblock {\em Integr. Equ. Oper. Theory}, 87:565--580, 2017.

\bibitem{KasparovAbsorbing}
G.~G. Kasparov.
\newblock Hilbert {$C^{\ast} $}-modules: theorems of {S}tinespring and
  {V}oiculescu.
\newblock {\em J. Operator Theory}, 4(1):133--150, 1980.

\bibitem{KucerovskyNgAUE}
Dan Kucerovsky and P.~W. Ng.
\newblock The corona factorization property and approximate unitary
  equivalence.
\newblock {\em Houston J. Math.}, 32(2):531--550, 2006.

\bibitem{LeeFirst}
Hyun~Ho Lee.
\newblock Proper asymptotic unitary equivalence in {$KK$}-theory and projection
  lifting from the corona algebra.
\newblock {\em J. Funct. Anal.}, 260(1):135--145, 2011.

\bibitem{LeeJFA2013}
Hyun~Ho Lee.
\newblock Deformation of a projection in the multiplier algebra and projection
  lifting from the corona algebra of a non-simple {$C^*$}-algebra.
\newblock {\em J. Funct. Anal.}, 265(6):926--940, 2013.

\bibitem{LinStableUniqueness}
Huaxin Lin.
\newblock Stable approximate unitary equivalence of homomorphisms.
\newblock {\em J. Operator Theory}, 47(2):343--378, 2002.

\bibitem{Lor-2018}
Jireh Loreaux.
\newblock Restricted diagonalization of finite spectrum normal operators and a
  theorem of arveson.
\newblock {\em Journal of Operator Theory}, 2019.
\newblock (to appear).

\bibitem{NgDoubleCommutantFinite}
Ping~W. Ng.
\newblock A double commutant theorem for the corona algebra of a {R}azak
  algebra.
\newblock {\em New York J. Math.}, 24:157--165, 2018.

\bibitem{Paschke}
William~L. Paschke.
\newblock {$K$}-theory for commutants in the {C}alkin algebra.
\newblock {\em Pacific J. Math.}, 95(2):427--434, 1981.

\bibitem{ThomsenAbsorption}
Klaus Thomsen.
\newblock On absorbing extensions.
\newblock {\em Proc. Amer. Math. Soc.}, 129(5):1409--1417, 2001.

\bibitem{Valette}
Alain Valette.
\newblock A remark on the {K}asparov groups {${\rm Ext}(A,\,B)$}.
\newblock {\em Pacific J. Math.}, 109(1):247--255, 1983.

\bibitem{Voiculescu}
Dan Voiculescu.
\newblock A non-commutative {W}eyl-von {N}eumann theorem.
\newblock {\em Rev. Roumaine Math. Pures Appl.}, 21(1):97--113, 1976.

\end{thebibliography}

\end{document}